\documentclass[twoside,12pt]{amsart}
\usepackage{amssymb}
\usepackage{amscd}
\usepackage{fullpage}
\usepackage{tikz-cd}

\usepackage[pdftex,colorlinks,hypertexnames=false,linkcolor=blue,urlcolor=blue,citecolor=blue]{hyperref}

\numberwithin{equation}{section}

\newtheorem{Theorem}{Theorem}[section]
\newtheorem{Proposition}[Theorem]{Proposition}

\def\iso{\cong}
\def\sub{\subseteq}

\def\C{\mathbb C}
\def\Z{\mathbb Z}

\def\cE{\mathcal E}
\def\cI{\mathcal I}
\def\cO{\mathcal O}
\def\cS{\mathcal S}
\def\cX{\mathcal X}

\def\g{{\mathfrak g}}
\def\l{\mathfrak l}
\def\m{\mathfrak m}
\def\n{\mathfrak n}
\def\q{\mathfrak q}
\def\t{\mathfrak t}
\def\sl{\mathfrak{sl}}
\def\u{\mathfrak u}
\def\z{\mathfrak z}

\def\bp{\overline{\mathfrak p}}

\def\isoto{\overset{\sim}{\longrightarrow}}

\def\modfd{\!\operatorname{-mod_{\operatorname{fd}}}}

\def\rA{\mathrm{A}}
\def\rB{\mathrm{B}}
\def\rC{\mathrm{C}}
\def\rD{\mathrm{D}}
\def\rE{\mathrm{E}}
\def\rF{\mathrm{F}}
\def\rG{\mathrm{G}}

\def\ba{\text{\boldmath$a$}}
\def\bb{\text{\boldmath$b$}}

\def\ad{\operatorname{ad}}

\def\ab{\operatorname{ab}}
\def\Ann{\operatorname{Ann}}
\def\Aut{\operatorname{Aut}}
\def\Comp{\operatorname{Comp}}

\def\Lie{\operatorname{Lie}}

\def\pt{\operatorname{pt}}
\def\wt{\operatorname{wt}}

\title
{On the variety of 1-dimensional representations of finite W-algebras in low rank}
\author{Jonathan Brown and Simon M.~Goodwin}
\address{ Department of Mathematics, Computer Science and Statistics,
                      State University of New York,
                      Oneonta, NY 13820, USA}
\email{Jonathan.Brown@oneonta.edu}
\address{School of Mathematics,
University of Birmingham,
Birmingham, B15 2TT,
UK}
\email{s.m.goodwin@bham.ac.uk}

\thanks{2010 {\it Mathematics Subject Classification}: 17B10, 17B37.}

\begin{document}

\begin{abstract}
Let $\g$ be a simple Lie algebra over $\C$ and let $e \in \g$ be nilpotent.
We consider the finite $W$-algebra $U(\g,e)$ associated to $e$ and the
problem of determining the variety $\cE(\g,e)$ of 1-dimensional representations
of $U(\g,e)$.  For $\g$ of low rank, we
report on computer calculations that have been used to determine the
structure of  $\cE(\g,e)$, and the action of the component group $\Gamma_e$
of the centralizer of $e$ on $\cE(\g,e)$.
As a consequence, we provide two examples where the nilpotent orbit of $e$ is induced,
but there is a 1-dimensional $\Gamma_e$-stable $U(\g,e)$-module which is not induced
via Losev's parabolic induction functor.  In turn this gives examples where there is a
``non-induced'' multiplicity free primitive ideal $I$ of $U(\g)$.
\end{abstract}

\maketitle

\section{Introduction}

Let $G$ be a simple algebraic group over $\C$, let $\g = \Lie G$ be the Lie algebra of
$G$, and let $e \in \g$ be nilpotent.
We write $U(\g,e)$ for the finite $W$-algebra
associated to $\g$ and $e$.  Finite $W$-algebras were introduced into the mathematical
literature by Premet in \cite{PrST} in 2002, and have subsequently attracted a lot
of research interest, we refer to \cite{Losurv} for a survey up to 2010.
The problem of understanding the 1-dimensional representations of $U(\g,e)$
has been of particular interest due the relationship with completely prime and
multiplicity free primitive ideals in $U(\g)$, and consequently quantizations of the algebra of regular functions on the
nilpotent orbit of $e$; see for example \cite{PrMF} and \cite{LoQO}, and the references therein.
This paper makes a contribution by giving an explicit
description of the variety of 1-dimensional representations of $U(\g,e)$
for $\g$ of low rank.

We introduce some notation required to discuss the background to and the contents of this paper further.
Let $I_c$ be the two-sided ideal of $U(\g,e)$ generated by the commutators
$uv - vu$ for $u,v \in U(\g,e)$, and let $U(\g,e)^{\ab} := U(\g,e)/I_c$.
The maximal spectrum $\cE = \cE(\g,e)$
of $U(\g,e)^{\ab}$
parameterizes the 1-dimensional representations of $U(\g,e)$.  As explained in
\cite[\S5.1]{PrT}, there is an action of the component group $\Gamma = \Gamma_e$ of the centralizer of $e$
in $G$ on $U(\g,e)^{\ab}$ and thus
on $\cE$.
The fixed point variety of
$\Gamma$ in $\cE$ is denoted by $\cE^\Gamma$ and is
identified with the maximal spectrum of $U(\g,e)^{\ab}_\Gamma := U(\g,e)^{\ab}/I_\Gamma$,
where $I_\Gamma$ is the two sided ideal of $U(\g,e)^{\ab}$ generated by all elements
of the form $u - \gamma \cdot u$ for $u \in U(\g,e)^{\ab}$ and $\gamma \in \Gamma$.
We let $\g^e$ denote the centralizer of $e$ in $\g$, and note that there is an action of $\Gamma$ on $\g^e/[\g^e,\g^e]$
as explained in \cite[\S5.1]{PrT}.  Let $c(e) := \dim(\g^e/[\g^e,\g^e])$ and $c_\Gamma(e) := \dim((\g^e/[\g^e,\g^e])^\Gamma)$.
We write $\cO_e \sub \g$ for the nilpotent orbit of $e$.

We briefly give an overview of previous research on 1-dimensional representations of $U(\g,e)$, and refer
to the
introductions to \cite{PrT} and \cite{PrMF} for a more detailed account.

In \cite[Conjecture 3.1]{PrEA}, Premet predicted that $\cE \ne \varnothing$ for all $\g$ and $e$,
i.e.\ that there exists  a 1-dimensional representation of $U(\g,e)$.
For $\g$ of classical type, Losev proved the existence of 1-dimensional representations of $U(\g,e)$
in \cite[Theorem 1.2.3]{LoQS}.  In \cite[Theorem~1.1]{PrCQ}, Premet gave a reduction of the conjecture to the case $e$ is rigid, and further
showed in \cite[Section~3]{PrCQ} that $\cE$ is finite for rigid $e$.  We recall that $e$ is
said to be {\em rigid}  if $\cO_e$ cannot be obtained via Lusztig--Spaltenstein
induction from a nilpotent orbit in a Levi subalgebra of $\g$.
Using computational methods, it was verified that $\cE \ne \varnothing$ when $e$ is rigid and $\g$ is of
types $\rG_2$, $\rF_4$, $\rE_6$ and $\rE_7$
by R\"ohrle, Ubly and the second author  in \cite[Theorem 1.1]{GRU};
a number cases for $\g$ of type $\rE_8$ and $e$ rigid are covered
in \cite[Remark~5.1]{GRU}, and some further cases are dealt with in the PhD thesis of Ubly, \cite{Ub}.
In \cite[Theorem~1.1.1]{LoPI}, Losev gave an alternative reduction to the case of rigid nilpotent orbits by
introducing a parabolic induction functor; we give a brief
account of this functor in Section~\ref{S:paraind}.  Further, Losev gave a method for finding 1-dimensional
representations of $U(\g,e)$  in \cite[Theorem~5.2.1]{LoPI}, which was
used to deal with a further $E_8$ case. Subsequently, this method was successfully exploited by Premet in \cite{PrMF}
for all cases where $\g$ is of exceptional type and $e$ is rigid, which allowed
him to verify that in fact $\cE^\Gamma \neq \varnothing$  for all $\g$ and $e$, see \cite[Theorem~A]{PrMF}.

We now recall known results on the structure of the varieties $\cE$ and $\cE^\Gamma$.
For $\g$ of type $\rA$, in which case $\Gamma$ is trivial, it was proved by Premet that
$U(\g,e)^{\ab}$ is a polynomial algebra of degree $c(e)$, so that $\cE \iso \C^{c(e)}$, in \cite[Theorem~3.3]{PrCQ}.
For $\g$ of other types, Premet and Topley consider $U(\g,e)^{\ab}_\Gamma$ when $e$ is an induced nilpotent element in
\cite[Theorems 2 and 4]{PrT}.  It is shown that $U(\g,e)^{\ab}_\Gamma$ is a polynomial
algebra of degree $c_\Gamma(e)$ for most cases, but seven cases for the pair $(\g,\cO_e)$ are excluded.
These seven cases are listed in \cite[Table 0]{PrT} and we note that $\g$ is of exceptional type in all of them.
Moreover, in the proof of \cite[Theorem 5]{PrT}, it was shown that in the non-excluded cases all the 1-dimensional representations
corresponding to points in $\cE^\Gamma$ are obtained via the parabolic induction from \cite[Theorem 1.1.1]{LoPI}.
In addition, in \cite[Theorems 1 and 4]{PrT}, it was proved that if $\cO_e$ is induced and nonsingular,
and not one of six of the cases from \cite[Table 0]{PrT}, then $U(\g,e)^{\ab}$ is a polynomial algebra of degree $c(e)$;
we recall that $e$ is {\em nonsingular} if it lies in a unique sheet of $\g$ and refer to the introduction to
\cite{PrT} for more details.

In this paper, we complete the picture for $\g$ of low rank by explicitly describing the structure of $\cE$ and $\cE^\Gamma$ in
cases not dealt with in \cite{PrT}.  More specifically, we deal with the cases where
$\cO_e$ is singular or listed in \cite[Table~0]{PrT} and $\g$ has rank 4 or less,
and also such cases for $\g$ of type $\rE_6$.  Our methods are computational and build on those used in \cite{GRU}.
It is interesting to observe that the structure of $\cE$ and the action of $\Gamma$ can already become quite complicated in
these low rank cases.

The cases of most interest are the two cases from \cite[Table 0]{PrT} for $\g$
of type $\rF_4$ and $e$ with Bala--Carter label $\rC_3(a_1)$, and for $\g$ of type $\rE_6$
and $e$ with Bala--Carter label $\rA_3+\rA_1$.
In these cases, we find that $\cE^\Gamma$ has two irreducible components and is not equidimensional:
one component is isomorphic to $\C$ and the other an isolated point.
From this we can deduce that there are $\Gamma$-stable 1-dimensional representations
of $U(\g,e)$ that are not induced using the parabolic induction functor from \cite[Theorem 1.1.1]{LoPI}.
The result that we require to make the deduction
is Proposition~\ref{P:noninduced} which implies that if
a 1-dimensional representation of $U(\g,e)$ is parabolically
induced, then it lies in a positive dimensional component of $\cE$.
We mention that under the standard embedding of $\g_{\rF_4}$
into $\g_{\rE_6}$ the nilpotent orbit $\rC_3(a_1)$ maps into the nilpotent orbit
$\rA_3+\rA_1$.

Next we recall that there
is a bijection between the $\Gamma$-orbits of finite dimensional irreducible representations of $U(\g,e)$
and  the primitive ideals of $U(\g)$ with associated variety $\overline \cO_e$.  This bijection is constructed by Losev, see \cite[Theorem~1.2.2]{LoQS}
and \cite[Theorem~1.2.2]{LoFD}, and we note that it can also be described in terms of Skryabin's equivalence from \cite{Sk}.
Under this bijection, the points in $\cE^\Gamma$ correspond to multiplicity free primitive ideals.
We refer for example to the introduction to \cite{PrMF} for the definition of multiplicity
free primitive ideals, and note that as explained there any multiplicity free primitive ideal is completely prime, but that the converse holds only
if $\g$ is of type $\rA$.
Further, we note that \cite[Theorem~6.4.1]{LoPI} shows that the parabolic induction functor
for finite dimensional modules of finite $W$-algebras intertwines in an appropriate
sense with the induction of primitive ideals.  We refer for example to
\cite[\S1.6]{PrT} for a discussion of induction of primitive ideals of universal
enveloping algebras; the definition is recalled in Section~\ref{S:paraind}.

The intertwining alluded to above forms a key step in the proof of \cite[Theorem~5]{PrT}, where it is shown that
if $e$ is induced, and not one of the seven excluded cases in \cite[Table~0]{PrT}, then any multiplicity
free primitive ideal of $U(\g)$ with associated variety $\overline \cO_e$ is induced from
a completely prime ideal $I_0$ of $U(\l)$ for some Levi subalgebra $\l$ of $\g$.  It
is said that \cite[Theorem~5]{PrT} may be considered as a generalization of
M{\oe}glin's theorem in type $A$ from \cite{Moe}.
In the discussion following \cite[Theorem~5]{PrT} it is speculated
that it is ``quite possible'' that the statement also holds for the cases listed in \cite[Table 0]{PrT}.
However, given that our computations give non-induced 1-dimensional representations
of $U(\g,e)$, we can deduce the following theorem regarding existence of ``non-induced'' multiplicity
free primitive ideals.

\begin{Theorem} \label{T:noninduced}
Let $\g$ be of type $\rF_4$ and $\cO_e$ with Bala--Carter label $\rC_3(a_1)$, or let $\g$ be of type $\rE_6$ and $\cO_e$ with Bala--Carter label
$\rA_3+\rA_1$. Then
there is a multiplicity free primitive ideal of $U(\g)$ with associated variety $\overline \cO_e$ that
cannot be induced from a primitive ideal of $U(\l)$ for any proper Levi subalgebra $\l$ of $\g$.
\end{Theorem}

In Section~\ref{S:irrcomp}, we recall a theorem of Premet, \cite[Theorem~1.2]{PrCQ}, which
is fundamental in understanding the set $\Comp(\cE)$ of irreducible components of $\cE$.
Then we are able to explain how this can be interpreted for the cases where we have
calculated $\cE$, and that this verifies low rank cases of a recent conjecture of Losev
from \cite[\S5.4]{LoOM}.

Lastly in the introduction, we mention that there are potential applications of our results to the
representation theory of modular Lie algebras.  This requires the reduction
modulo $p$ of finite $W$-algebras introduced by Premet, see for example
\cite{PrCQ}.  The applications would be in the study of minimal dimensional
representations of reduced enveloping algebras, and we note that the nature
of the computations put some restrictions on the characteristic.

An outline of this paper is as follows.  In Section~\ref{S:background} we recall the background
on finite $W$-algebras that we require to explain our algorithm and results.  An outline of the
algorithm is presented in Section~\ref{S:algorithm} and then results of the computations are
explained in Section~\ref{S:results}.  In Section~\ref{S:paraind}, we give a recollection
of the parabolic induction from \cite{LoPI},  prove Proposition~\ref{P:noninduced} and
then deduce Theorem~\ref{T:noninduced}.  Lastly in Section~\ref{S:irrcomp} we relate our
results to Premet's map on irreducible components.

\subsection*{Acknowledgments}  We thank
A.~Premet and L.~Topley for very helpful discussions and email
correspondence about this work.  The first author is grateful
for travel funds from SUNY, Oneonta to visit the University of Birmingham.

\section{Background on \texorpdfstring{$U(\g,e)$}{U(g,e)} and its PBW basis} \label{S:background}

We recall the relevant facts about $U(\g,e)$ necessary to calculate $\cE$.
This is mostly taken from \cite{PrST} or \cite{PrEA}, and we only reference results
not contained there.

 \subsection{Definition of \texorpdfstring{$U(\g,e)$}{U(g,e)}}
We continue to use the notation given in the introduction, so that $G$ is a simple algebraic group over $\C$, and $e$ is a nilpotent element in the
Lie algebra $\g$ of $G$.  We write $\Gamma := G^e/(G^e)^\circ$ for the component group
of the centralizer of $e$ in $G$.  Also we fix $(\cdot\,,\cdot)$ to be a non-degenerate invariant
symmetric bilinear form on $\g$ (for example the Killing form) and let $\chi := (e,\cdot) \in \g^*$.
We embed $e$ into an $\sl_2$-triple $(e,h,f)$.  A maximal toral subalgebra $\t$ of $\g$ is said to
be {\em compatible} with $(e,h,f)$ it $h \in \t$ and $\t$ contains a maximal toral subalgebra $\t^e$
of the centralizer $\g^e$ of $e$ in $\g$.  We fix $\t$ to be a compatible maximal toral subalgebra
of $\g$.  We denote the restricted root system of $\g$ with respect to $\t^e$ by $\Phi^e$, as defined in \cite[Section 3]{BG}.
Also we define the {\em normalizer} of $e$ in $\t$ to be $\t^{[e]} := \{x \in \t \mid [x,e] \in \C e\}$, and note that this is equal to $\t^e \oplus \C h$.

The $\ad h$ eigenspace decomposition determines a grading
$$
\g = \bigoplus_{j \in \Z} \g(j),
$$
where $\g(j) := \{ x \in \g \mid [h,x] = jx\}$.
We define a symplectic form $\langle \cdot,\cdot \rangle$ on $\g(-1)$ by
$\langle x,y \rangle := \chi([x,y])$, for $x,y \in \g(-1)$.
Let $\l$ be an isotropic subspace of $\g(-1)$ with respect to the form $\langle \cdot , \cdot \rangle$
and let $\l^\perp := \{x \in\g(-1) \mid \langle x,y \rangle = 0\}$.
Also let $\l'$ be a subspace of $\g(-1)$ which is complementary to $\l$.
We may, and do, assume that
$\l$ and $\l'$ are stable under the adjoint action of $\t$ as this is suitable for our
calculations.

Let $\m := \l \oplus \bigoplus_{i \le -2} \g(i)$
and $\n := \l^\perp \oplus \bigoplus_{i \le -2} \g(i)$, which
are nilpotent subalgebras of $\g$ stable under the adjoint action of $\t$.
Then $\chi$ restricts to a character of $\m$ and we
let $I$ be the left ideal of $U(\g)$ generated by $\{x - \chi(x) \mid x \in \m \}$.
There is an adjoint action of $\n$ on $U(\g)/ I$ and the {\em finite $W$-algebra}
is defined to be
$$
U(\g,e) := (U(\g)/I)^\n = \{ u + I \in U(\g)/I \mid [x,u] \in I \text{ for all } x \in \n \}.
$$
We note that the definition of $U(\g,e)$ below depends on the choice of $\l$, but only up to isomorphism
thanks to \cite[Theorem 4.1]{GG}.

\subsection{The component group \texorpdfstring{$\Gamma$}{Gamma}}
Let $C$ be the centralizer of $h$ in $G$, so that $\Lie C = \g(0)$, and let $C^e$ be the centralizer
of $e$ in $C$.  The component group $\Gamma$ is isomorphic to $C^e/(C^e)^\circ$.  For the case
$\l = 0$,
there is an adjoint action of $C^e$ on $U(\g,e)$.  In the cases considered in this paper, it turns
out that we can choose lifts in $C^e$ of the elements of $\Gamma$, which generate a subgroup of $C^e$ isomorphic to
$\Gamma$.  Thus, in this paper, we allow ourselves to speak of an action of $\Gamma$ on $U(\g,e)$, though we do not claim that
there is an action of $\Gamma$ on $U(\g,e)$ in general.  Also  if $C^e$ is connected, so that $\Gamma$ is trivial,
then we can still speak of the action of $\Gamma$ on $U(\g,e)$, when $\l$ is chosen to be a nonzero
isotropic subspace of $\g(-1)$.
We note that this action of
$\Gamma$ on $U(\g,e)$ induces an action on $U(\g,e)^{\ab}$, which is the same
as the action considered in the introduction.

\subsection{PBW bases for \texorpdfstring{$U(\g)$}{U(g)} and \texorpdfstring{$U(\g,e)$}{U(g,e)}}
Let $\bp := \mathfrak{l}' \oplus \bigoplus_{i \ge 0} \g(i)$. Note that if $\l \ne 0$, then $\bp$ is not
necessarily a subalgebra of $\g$;
it is just a $\t$-stable subspace of $\g$.
We fix a Chevalley basis $\{x_1,\dots,x_m,y_1,\dots,y_s\}$ of $\g$ with respect to $\t$
such that $\{x_1,\dots,x_m\}$ is a basis of $\bp$ and $\{y_1,\dots,y_s\}$ and basis for $\m$.
This is chosen so that $\{y_1,\dots,y_l\} \sub \{y_1,\dots,y_s\}$
is a minimal generating set of $\m$.
We have that $x_1,\dots,x_m$ are weight vectors
for $\t^{[e]}$: we write $n_i$ for the eigenvalue of $h$ and $\alpha_i \in \Phi^e \cup \{0\}$
for the $\t^e$-weight of $x_i$.

We obtain a PBW basis of $U(\g)$ with elements $x^\ba y^\bb := x_1^{a_1} \dots x_m^{a_m} y_1^{b_1} \dots y_s^{b_s}$
for $\ba \in \Z_{\ge 0}^m$, $\bb \in \Z_{\ge 0}^s$.
This PBW basis can be used to give an isomorphism of vector spaces $S(\bp) \isoto U(\g)/I$
defined by $x^\ba \mapsto x^\ba + I$;
this isomorphism is helpful when making calculations, as it allows to us to work in the vector space $S(\bp)$
instead of the quotient space $U(\g)/I$.

To give a PBW basis for $U(\g,e)$,
we fix a basis $\{z_1,\dots,z_r\}$ of $\g^e$, consisting of $\t^{[e]}$-weight vectors;
chosen so that $\{z_1,\dots,z_p\} \sub \{z_1, \dots, z_r\}$ is a minimal generating set of $\g^e$.
Let $m_i$ be the $\ad h$-eigenvalue and $\beta_i \in \Phi^e \cup \{0\}$ be
the $\t^e$-weight of $z_i$.

For $\ba \in \Z_{\ge 0}^m$ we define $|\ba| := \sum_{i=1}^m a_i$ to be the {\em total degree},  $|\ba|_e := \sum_{i=1}^m (n_i+2)a_i$ to be
the {\em Kazhdan degree},
and $\wt(\ba) := \sum_{i=1}^m a_i \alpha_i \in \Z\Phi^e$ to be the $\t^e$-weight of $x^\ba$.
We make similar definitions for $\bb \in \Z_{\ge 0}^r$, i.e.\
we define $|\bb| := \sum_{i=1}^r b_i$, $|\bb|_e := \sum_{i=1}^r b_i(m_i+2)$,
and $\wt(\bb) := \sum_{i=1}^r b_i \beta_i$.

By the PBW theorem for $U(\g,e)$, there is a (non-unique) vector space monomorphism
\begin{equation} \label{e:Theta}
\Theta : \g^e \to U(\g,e)
\end{equation}
equivariant under the action of $\t^e$ and $\Gamma$, and
such that $\{\Theta(z_i) \mid i = 1,...,r\}$ generates $U(\g,e)$ and the PBW
monomials
$$
\{\Theta(z_1)^{b_1} \cdots \Theta(z_r)^{b_r} \mid \bb \in \Z_{\ge 0}^r\}
$$
form a basis of $U(\g,e)$.
Moreover, $\Theta$ can be chosen so that
$$
\Theta(z_i)
=  z_i + \sum_{\substack{\ba \in \Z_{\ge 0}^m, \\ |\ba|_e \le n_i + 2, \\ \wt(\ba) = \beta_i}} \lambda_{\ba}^i x^\ba + I,
$$
where $\lambda_\ba \in \C$ satisfy $\lambda_\ba = 0$ whenever $|\ba|_e = n_i+2$ and $|\ba| = 1$.

We abbreviate notation and write $\Theta_i := \Theta(z_i)$, and $\Theta^\bb := \Theta_1^{b_1} \cdots \Theta_r^{b_r}$ for
$\bb = (b_1,\dots,b_r) \in \Z_{\ge 0}^r$. We also write $z^\bb := z_1^{b_1} \dots z_r^{b_r} \in U(\g^e)$ for $\bb \in \Z_{\ge 0}^r$.

\subsection{Commutators in \texorpdfstring{$U(\g,e)$}{U(g,e)}}
We recall the form of the commutators of the PBW generators of $U(\g,e)$, and explain how
these can be used to determine the variety of 1-dimensional representations $\cE$ of $U(\g,e)$.

For our generating set $\Theta_1, \dots, \Theta_r$ of $U(\g,e)$,
the commutators are of the form
\begin{equation} \label{e:comms}
[\Theta_i,\Theta_j] = \Theta([z_i,z_j]) + \sum_{\substack{\bb \in \Z_{\ge 0}^r, \\|\bb|_e \le m_i+m_j+2, \\ \wt(\bb) = \beta_i + \beta_j}} \nu_\bb^{i,j} \Theta^\bb,
\end{equation}
 For convenience we incorporate $\Theta([z_i,z_j])$ into this sum and write
$[\Theta_i,\Theta_j] = \sum_{\bb \in \Z_{\ge 0}^r} \nu_\bb^{i,j} \Theta^\bb$.

A 1-dimensional representation of $U(\g,e)$ is determined by $(\theta_1,\dots,\theta_r) \in \C^r$ such that
\begin{equation} \label{e:1dimcond}
\sum_{\bb \in \Z_{\ge 0}^r} \nu_\bb^{i,j} \theta^\bb = 0
\end{equation}
for all $i,j \in \{1,\dots,r\}$; here $\theta^\bb := \theta_1^{b_1} \dots \theta_r^{b_r}$.  So that
we can identify $\cE$ with the variety formed by such $(\theta_1,\dots,\theta_r) \in \C^r$.

In general the most computationally expensive part in our algorithm for determining $\cE$ is finding the commutators $[\Theta_i, \Theta_j]$.
The results in \cite[\S3]{GRU} allow us only find the minimal number of commutators in order to solve the equations \eqref{e:1dimcond}.
First, we note that for $(\theta_1,\dots,\theta_r) \in \C^r$ to give
a 1-dimensional representation we must have $\theta_i = 0$ if $\beta_i \ne 0$.
In particular, this implies that we do not need to find the commutators $[\Theta_i, \Theta_j]$ when $\beta_i \neq - \beta_j$.
Also we do not need to find the commutators $[\Theta_i, \Theta_j]$ when neither $z_i$ or $z_j$ is in our minimal generating set
$\{z_1,\dots,z_p\}$ of $\g^e$, as these commutators can be deduced from the others.
Let
$I := \{i \in {1,\dots,r} \mid \beta_i = 0\}$ and
\begin{equation} \label{e:mincomms}
J := \{(i,j) \in \{1,\dots,p\} \times \{1,\dots,r\} \mid \beta_j = -\beta_i\}.
\end{equation}
Summarizing the discussion above, \cite[Proposition 3.5]{GRU} says $\cE$ is completely determined by solutions to the equations
\begin{equation} \label{e:mincomms2}
\sum_{\bb \in \Z_{\ge 0}^I} \nu_\bb^{i,j} \theta^\bb = 0
\end{equation}
for $(i,j) \in J$, where by $\Z_{\ge 0}^I$ we mean the subset of $\Z_{\ge 0}^r$ of
those $\bb \in \Z_{\ge 0}^r$ for which $b_i = 0$ for $i \not\in I$.

\section{The Algorithm} \label{S:algorithm}

Our algorithm for determining $\cE$ and $\cE^\Gamma$ is based on the algorithm in \cite[$\S4$]{GRU}, though we have incorporated
some significant improvements, which allow us to deal with more complicated cases.
This includes taking account of the action of $\Gamma$ in the definition of the map $\Theta$ from \eqref{e:Theta}, which
tends to make the commutators simpler.  Also we work directly with the Chevalley basis for $\g$, which appears
to improve the efficiency of the algorithm.  Many other optimizations are included,
but in the outline of the algorithm below we do not include the details of all of these for simplicity.

We have programmed the computational steps of the algorithm
in the computer algebra language GAP, \cite{GAP}, so that the internal functions for Lie algebras can be used.
Custom classes and functions were programmed to do calculations in $U(\g)$ (as opposed to the built-in universal enveloping
algebra functions), as this was more convenient for our calculations.

In the description of the algorithm below, we use the notation introduced in the previous section.
Also we use {\em italics} to give some comments
to help with understanding.

\smallskip
\noindent
{\bf Input:} A simple Lie algebra $\g$ over $\C$ with maximal toral subalgebra $\t$, a nilpotent element $e$ of $\g$, and
generators $\{g_1,\dots,g_a\}$ for $\Gamma$ viewed as a subgroup of $\Aut(\g)$.

\noindent
{\em The nilpotent element $e$ input must be compatible with $\t$: it is chosen by using pyramids
from \cite[Sections 5 and 6]{EK} for classical types, and
for exceptional Lie algebras an orbit representative for each orbit can be found in \cite[Section 11]{LT}.
The lifts of the generators of $\Gamma$ are found as elements of $C^e$ as explained in the previous section, and
then viewed as elements of $\Aut(\g)$: for classical Lie algebras, explicit
formulas can be given in terms of the Dynkin pyramids by extending the methods used in \cite[Section 6]{Bro};
and for exceptional types explicit generators can be found in \cite[Section 11]{LT}.}

\smallskip
\noindent
{\bf Steps in the algorithm:}

\begin{enumerate}

\item
Find an $\mathfrak{sl}_2$-triple $(e,h,f)$,

\item Make a choice $\l$ of an isotropic subspace of $\g(-1)$
stable under the adjoint action of $\Gamma$,
and a complement $\l'$ of $\l$ in $\g(-1)$.  \\
{\em In the cases that we consider, we take $\l = 0$ when $\Gamma$ is nontrivial, and $\l$ to be Lagrangian
when $\Gamma$ is trivial.}

\item Choose bases $\{x_1,\dots,x_m\}$ of $\bp$ and $\{y_1,\dots,y_s\}$ of $\m$.
The basis of $\m$ is chosen so that $\{y_1,\dots,y_l\}$ is a minimal subset that generates $\m$ as a Lie algebra.

\item
Calculate a basis $\{z_1,\dots,z_r\}$ for $\g^e$ such that each $z_i$ is a weight vector for $\t^{[e]}$, and such that $\{z_1,\dots,z_p\}$
is a minimal generating set of $\g^e$.

\item For each basis element $z_i$ of $\g^e$ we find $\Theta_i \in U(\g,e)$ using the following steps:

\begin{enumerate}
\item
Determine the set
$$
\tilde A_i := \{\ba \in \Z_{\ge 0}^m \mid \wt(\ba) = \beta_i, \text{ and } |\ba|_e \le n_i \text{ or } |\ba|_e = n_i+2, |\ba| > 1\}.
$$
  Let $T$ be the maximal torus of $G$ whose Lie algebra is $\t$.
  For $\sigma \in \Gamma\cap T$, let $c_\sigma \in \C$ be such that $\sigma (z_i) = c_\sigma z_i$ and let
$$
    A_i := \{ \ba \in \tilde A_i \mid \sigma(x^\ba) = c_\sigma x^\ba \text{ for all } \sigma \in \Gamma \cap T\}.
$$
Enumerate $A_i = \{\ba^1,\dots,\ba^M\}$. \\
For indeterminants $t_1,\dots,t_M$, let
$\tilde \Theta_i := z_i + \sum_{j=1}^M t_j x^{\ba^j}+I$.

\item
For $k=1,\dots,l$, calculate $[y_k,\tilde \Theta_i]$ in the form
$\sum_{\ba \in \Z_{\ge 0}^m} \lambda_\ba(t_1,\dots,t_M) x^\ba + I$, where
$\lambda_\ba(t_1,\dots,t_M)$ is a linear combination of $t_1,\dots,t_M$. \\
{\em Of course only finitely many of the $\lambda_\ba$ are nonzero.}

\item
Determine a solution to the system  of linear equations $\lambda_\ba(t_1,\dots,t_M) = 0$, for $\ba \in \Z_{\ge 0}$. \\
{\em This system will be very large in general, so finding a solution requires employing methods from linear algebra
involving sparse matrices.}

\item Let $c_1,\dots,c_M$ be this solution and
set $\Theta_i := z_i + \sum_{j=1}^M c_j x^{\ba^j}$.

\end{enumerate}

\item
Calculate the set of commutators $[\Theta_i,\Theta_j]$ for $(i,j) \in J$,
where $J$ is defined in \eqref{e:mincomms}.  These commutators are calculated
in the form given in \eqref{e:comms} by using the following procedure to write an
element $u \in U(\g,e)$ (which we assume to be a $\t^e$-weight vector of weight $\gamma \in \Z\Phi^e$)
as a linear combination of the PBW basis $\{\Theta^\bb \mid \bb \in \Z_{\ge 0}^r\}$.
\begin{enumerate}
\item Write $u = \sum_{\ba \in \Z_{\ge 0}^r} \mu_\ba x^\ba + I$ $\in U(\g,e)$,
let $R(u)$ be maximal in $\{ |\ba|_e \mid \mu_\ba \ne 0\}$, let $S(u)$
be minimal in $\{ |\ba| \mid |\ba|_e = R(u)\}$. \\
Let $M(u) := \{\ba \in \Z_{\ge 0}^m \mid  |\ba|_e = R(u), |\ba| = S(u)\}$.
\item From the PBW theorem for $U(\g,e)$ it follows that
$\sum_{\ba \in M(u)} \mu_\ba x^\ba$ can be expressed in the form $\sum_{\bb \in \Z_{\ge 0}^r} \delta_\bb z^\bb$,
where $\lambda_\bb \in \C$, and the sum is over those $\bb \in \Z_{\ge 0}^r$ with $|\bb| = S(u)$ and $\t^e$-weight $\gamma$.
The coefficients $\delta_\bb$ are found by solving a system of linear equations.
\item We consider
$$
v := u - \sum_{\bb \in \Z_{\ge 0}^r} \delta_\bb \Theta(z)^\bb.
$$
We have that $R(v) < R(u)$, or $R(v) = R(u)$ and $S(v) > S(u)$.
Thus we can recursively continue to subtract terms and obtain an
expression for $u$ as a linear combination of the PBW basis $\{\Theta^\bb \mid \bb \in \Z_{\ge 0}^r\}$.
\end{enumerate}

\item

Use the commutators found in the previous step to find the values
$\theta_1,\dots,\theta_r \in \C$ that are solutions of
$$
\sum_{\ba \in \Z_{\ge 0}^r}  \nu_\ba^{i,j} \theta^\ba = 0.
$$
In fact we consider the reduced system of equations given by \eqref{e:mincomms2}. \\
{\em Finding the solution to a system of non-linear equations is done with ad hoc methods
as the degrees of the equations are low in the examples we are considering.}

\item Use the action of $\Gamma$ on $\{z_1, \dots, z_r\}$ to calculate the action of $\Gamma$ on $\cE$
and to determine $\cE^\Gamma$.
\end{enumerate}

\smallskip
\noindent
{\bf Output:} The description of $\cE$ and $\cE^\Gamma$ in terms of the values $\theta_1,\dots,\theta_r$ that can be taken by
$\Theta_1,\dots,\Theta_r$ when $U(\g,e)$ acts on a 1-dimensional module.

\section{Results} \label{S:results}

We have run our program to determine
$\cE$ and $\cE^\Gamma$ for all cases of induced orbits
that are not covered by \cite[Theorems 2 and 4]{PrT} and for which $\g$ has rank 4 or less, or is of type $\rE_6$.
The calculations were done on a typical desktop computer.  Nearly all of the computationally intensive steps are highly parallelizable,
so with access to a large enough distributed system of computers it is plausible that similar calculations would be able deal with
more of the orbits in exceptional Lie algebras from \cite[Table 0]{PrT}.

We label each of the cases that we have calculated by the type of $\g$, and a label for the nilpotent orbit: for exceptional types we
give the Bala--Carter label and for classical types we give the partition giving the Jordan type.
We present the results of these
computations below.  For each case we explain the structure of $\cE$, and then the action of
$\Gamma$ on $\cE$.  To do this we have implicitly
introduced some coordinates on the irreducible components.  Then
we state the structure of $\cE^\Gamma$; we note this is known from
\cite[Theorems 2 and 4]{PrT} in all except the cases $(\rF_4,\rC_3(a_1))$ and $(\rE_6,\rA_3+\rA_1)$,
but we include it for completeness.

\subsection*{\texorpdfstring{$\rG_2: \rG_2(a_1)$}{G2: G2(a1)}} $ $
\begin{itemize}
\item $\cE$ has four components isomorphic to $\C$ which pairwise meet at a common point.
\item $\Gamma \iso S_3$ fixes one of the components and the other three are permuted by $\Gamma$.
\item $\cE^\Gamma \iso \C$.
\end{itemize}

\subsection*{\texorpdfstring{$\rF_4: \rC_3(a_1)$}{F4: C3(a1)}} $ $
\begin{itemize}
\item $\cE$ has four components: one isomorphic to $\C$ and three isolated points.
\item $\Gamma \iso S_2$.  The 1-dimensional component and one of the
points are fixed by $\Gamma$ and the other two points are transposed.
\item $\cE^\Gamma \iso \C \sqcup \{\pt\}$.
\end{itemize}

\subsection*{\texorpdfstring{$\rF_4: \rF_4(a_1)$}{F4: F4(a1)}} $ $
\begin{itemize}
\item $\cE$ has two components isomorphic to $\C^3$ and their intersection is isomorphic to $\C^2$.
\item $\Gamma \iso S_2$.
One component is fixed by $\Gamma$ and the other is reflected
in the intersection of the two components.
\item $\cE^\Gamma \iso \C^3$.
\end{itemize}

\subsection*{\texorpdfstring{$\rF_4: \rF_4(a_2)$}{F4: F4(a2)}} $ $
\begin{itemize}
\item $\cE$ has two components isomorphic to $\C^2$ and their intersection is isomorphic to $\C$.
\item $\Gamma \iso S_2$.
One component is fixed by $\Gamma$ and the other is reflected
in the intersection of the two components.
\item $\cE^\Gamma \iso \C^2$.
\end{itemize}

\subsection*{\texorpdfstring{$\rF_4: \rF_4(a_3)$}{F4: F4(a3)}} $ $
\begin{itemize}
\item $\cE$ has three components isomorphic to $\C^2$ and five components isomorphic to $\C$: all pairwise intersections are a common point.
\item $\Gamma \iso S_4$.  The three components isomorphic to $\C^2$ are permuted by the
quotient of $\Gamma$ isomorphic to $S_3$.   One of the components isomorphic to $\C$ is fixed by $\Gamma$.  The other
four components isomorphic to $\C$ are permuted by $\Gamma$ in a natural way.
\item $\cE^\Gamma \iso \C$.
\end{itemize}

\subsection*{\texorpdfstring{$\rE_6: \rA_3+\rA_1$}{E6:A3+A1}}  $ $
\begin{itemize}
\item $\cE$ has two components: one isomorphic to $\C$, and the other a point.
\item $\Gamma$ is trivial.
\item $\cE^\Gamma \iso \C \sqcup \{\pt\}$.
\end{itemize}

\subsection*{\texorpdfstring{$\rE_6: \rE_6(a_3)$}{E6: E6(a3)}} $ $
\begin{itemize}
\item $\cE$ has two components, one isomorphic to $\C^4$, the other isomorphic to $\C^3$, and their intersection
is isomorphic to $\C^2$
\item $\Gamma \iso S_2$.  The component isomorphic to $\C^3$ is fixed by $\Gamma$.  The nonidentity element
acts on the component isomorphic to $\C^4$ via $(x,y,z,w) = (-x,-y,z,w)$, where the
intersection of the two components is $\{(0,0,z,w) \mid z,w \in \C\}$.
\item $\cE^\Gamma \iso \C^3$.
\end{itemize}

\subsection*{\texorpdfstring{$\rE_6: \rD_4(a_1)$}{E6: D4(a1)}} $ $
\begin{itemize}
\item $\cE$ has 5 components, four of which are isomorphic to $\C^2$ and the other is isomorphic to $\C$.
All of the components pairwise intersect in a common point.
\item $\Gamma \iso S_3$.  Three of the components isomorphic to $\C^2$ are permuted by $\Gamma$ in the natural way.
The action of $\Gamma$ on the fourth component isomorphic to $\C^2$ is via the irreducible 2-dimensional
representation of $\Gamma$.  The component isomorphic to $\C$ is fixed by $\Gamma$.
\item $\cE^\Gamma \iso \C$.
\end{itemize}

\subsection*{\texorpdfstring{$\rC_2: (2,2)$}{C2: (2,2)}} $ $
\begin{itemize}
\item $\cE$ has two components isomorphic to $\C$, which intersect in a point.
\item $\Gamma \iso S_2$.  One component is fixed by $\Gamma$ and the other is reflected in the intersection
of the two components.
\item $\cE^\Gamma \iso \C$.
\end{itemize}

\subsection*{\texorpdfstring{$\rC_3: (4,2)$}{C3: (4,2)}} $ $
\begin{itemize}
\item $\cE$ has two components isomorphic to $\C^2$ and their intersection is isomorphic to $\C$.
\item $\Gamma \iso S_2$.
One component is fixed by $\Gamma$ and the other is reflected
in the intersection of the two components.
\item $\cE^\Gamma \iso \C^2$.
\end{itemize}

\subsection*{\texorpdfstring{$\rB_3: (5,1,1)$}{B3: (5,1,1)}}
\begin{itemize}
\item $\cE$ has two components isomorphic to $\C^2$ and their intersection is isomorphic to $\C$.
\item $\Gamma \iso S_2$.
One component is fixed by $\Gamma$ and the other is reflected
in the intersection of the two components.
\item $\cE^\Gamma \iso \C^2$.
\end{itemize}

\subsection*{\texorpdfstring{$\rC_4: (4,2,2)$}{C4: (4,2,2)}} $ $
\begin{itemize}
\item $\cE$ has two components: one isomorphic to $\C^2$, the other isomorphic to $\C$, and they intersect in a point.
\item $\Gamma \iso S_2$.  The component isomorphic to $\C^2$ is fixed by $\Gamma$.
The component isomorphic to $\C$ is reflected in the intersection of the two components.
\item $\cE^\Gamma \iso \C^2$.
\end{itemize}

\subsection*{\texorpdfstring{$\rC_4: (6,2)$}{C4: (6,2,2)}} $ $
\begin{itemize}
\item $\cE$  has two components isomorphic to $\C^3$,
and their intersection is isomorphic to $\C^2$.
\item $\Gamma \iso S_2$.
One component is fixed by $\Gamma$ and the other is reflected
in the intersection of the two components.
\item $\cE^\Gamma \iso \C^3$.
\end{itemize}

\subsection*{\texorpdfstring{$\rB_4: (7,1,1)$}{B4: (7,1,1)}} $ $
\begin{itemize}
\item $\cE$  has two components isomorphic to $\C^3$,
and their intersection is isomorphic to $\C^2$.
\item $\Gamma \iso S_2$.
One component is fixed by $\Gamma$ and the other is reflected
in the intersection of the two components.
\item $\cE^\Gamma \iso \C^3$.
\end{itemize}

\subsection*{\texorpdfstring{$\rB_4: (5,3,1)$}{B4: (5,3,1)}} $ $
\begin{itemize}
\item $\cE$ has three components isomorphic to $\C^2$,
where each pair of components intersects in a variety isomorphic to $\C$ and the
three components intersect in a point.
\item $\Gamma \iso S_2 \times S_2$.  Denote the components by $A$, $B$ and $C$ and let $r,s$ be generators of $\Gamma$.
The component $A$ is fixed by $\Gamma$; while $r$ fixes $B$ and acts on $C$ by the reflection in $A \cap C$ and
$s$ fixes $C$ and acts on $B$ by the reflection in $A \cap B$.
\item $\cE^\Gamma \iso \C^2$.
\end{itemize}

\subsection*{\texorpdfstring{$\rD_4: (3,3,1,1)$}{D4: (3,3,1,1)}} $ $
\begin{itemize}
\item $\cE$ has two components: one isomorphic to $\C^2$ and the other isomorphic to $\C$, and they
intersect in a point.
\item $\Gamma \iso S_2$.  The component isomorphic to $\C$ is fixed by $\Gamma$, and the non-identity
element of $\Gamma$ acts on the component isomorphic to $\C^2$ by $(x,y) \mapsto (-x,-y)$.
\item $\cE^\Gamma \iso \C$.
\end{itemize}

\section{Parabolic induction for finite \texorpdfstring{$W$}{W}-algebras} \label{S:paraind}

The goal of this section is to prove Theorem~\ref{T:noninduced}.  We need to provide some
preliminaries beginning with the Lusztig--Spaltenstein induction of nilpotent orbits.

Let $\g'$ be a Levi subalgebra of $\g$, with $\g' \ne \g$, and let $\q = \g' \oplus \u$ be a parabolic subalgebra of
$\g$ with Levi factor $\g$ and nilradical $\u$.
Lusztig-Spaltenstein induction provides a way to induce a nilpotent orbit $\cO'$ in $\g'$ to a
nilpotent orbit $\cO$ in $\g$; it is defined by declaring that $\cO$ is the unique orbit such that
$(\cO' + \u) \cap \cO$ is open in $\cO' + \u$.
We fix a nilpotent orbit $\cO'$ in $\g'$, let $\cO$ be the nilpotent orbit in $\g$ obtained from
$\cO'$ by Lusztig--Spaltenstein induction, and let
$e' \in \cO'$ and $e \in \cO$.  Let $\cE'$ and $\cE$ be the varieties of one-dimensional
$U(\g',e')$-modules and $U(\g,e)$-modules respectively.

In \cite[Theorem 1.2.1]{LoPI} Losev introduced a {\em parabolic induction} functor
$$
\rho_\q^\g : U(\g',e')\modfd \to U(\g,e)\modfd
$$
from the category
of finite dimensional $U(\g',e')$-modules to the category of finite dimensional $U(\g,e)$-modules.
Moreover, $\rho_\q^\g$ is dimension preserving, so
determines a morphism $\cE' \to \cE$, which by \cite[Theorem 6.5.2]{LoPI} is a finite morphism.

We are now ready to state and prove Proposition \ref{P:noninduced}, which is the key result
we require to prove Theorem~\ref{T:noninduced}.  In the statement $\operatorname{rank} \g$ denotes the rank of $\g$, and
$\operatorname{ssrank} \g'$ denotes the semisimple rank of $\g'$.

\begin{Proposition} \label{P:noninduced}
Let $M'$  be a be a 1-dimensional $U(\g',e')$-module corresponding
to a point in $\cE'$ and let $M = \rho_\q^\g(M')$.
Then the point of $\cE$ corresponding to $M$ lies in an
irreducible component of $\cE$ of dimension at least $\operatorname{rank} \g - \operatorname{ssrank} \g'$.
\end{Proposition}

\begin{proof}
We have $\g' = [\g',\g']+ \z(\g')$, where $\z(\g')$ denotes the centre of $\g'$.  From the definition of $U(\g',e')$
it is straightforward to see that $U(\g',e') \iso U([\g',\g'],e') \otimes S(\z(\g'))$. Let $\sigma' : U(\g',e) \to \C$ be the
representation of $U(\g',e)$ corresponding to $M'$.  Given any character $\zeta : S(\z(\g')) \to \C$
we let $\sigma'_\zeta : U(\g',e) \to \C$ be the 1-dimensional representation with
$\sigma'_\zeta(u \otimes z) := \sigma'(u \otimes z)\zeta(z)$, and let $M'_\zeta$ be the corresponding 1-dimensional $U(\g',e)$-module.
We identify each $M'_\zeta$ for $\zeta$ a character of $S(\z(\g'))$ with a point of $\cE'$.
The closure of $\{M'_\zeta \mid \zeta$ is a character of $S(\z(\g'))\}$ is an irreducible subvariety of $\cE'$ of dimension
$\operatorname{rank} \g - \operatorname{ssrank} \g'$.  Since the parabolic induction functor gives a finite morphism $\cE' \to \cE$, the image
of this irreducible subvariety in $\cE$ has the same dimension.  The closure of this image is also irreducible and contains the point corresponding to $M$.
Thus $M$ lies in an irreducible component of $\cE$ of dimension at least $\dim \z(\g') = \operatorname{rank} \g - \operatorname{ssrank} \g'$.
\end{proof}

Before moving on to prove Theorem~\ref{T:noninduced}, we need to recall
Losev's map of ideals
and
how Losev's
parabolic induction functor intertwines with the induction of ideals.

We write $\cdot^\dagger : \operatorname{Id}(U(\g,e)) \to
\operatorname{Id}(U(\g))$ for Losev's map from (two-sided) ideals of $U(\g,e)$ to (two-sided) ideals of
$U(\g)$, see \cite[Theorem~1.2.2]{LoQS}.  By parts (v) and (vi) of that theorem, the restriction of
$\cdot^\dagger$ to the set of ideals of $U(\g,e)$ of finite codimension
maps into the set of ideals of $U(\g)$ with associated variety equal
to $\overline{\cO}$.  Further by  \cite[Theorem~1.2.2(viii)]{LoQS}, the restriction of
$\cdot^\dagger$ to the set of primitive ideals of $U(\g,e)$ of finite codimension
maps surjectively onto the set of primitive ideals of $U(\g)$ with associated variety equal
to $\overline{\cO}$, and by \cite[Conjecture~1.2.1]{LoFD} (which is
deduced from \cite[Theorem~1.2.2]{LoFD}) the fibres are $\Gamma$-orbits.

We recall the definition of parabolic induction of from ideals
of $U(\g')$ to ideals of $U(\g)$.
Given a ideal $I'$ of $U(\g')$ we let $\cI_\q^\g(I')$
be the largest two-sided ideal of $U(\g)$ contained in the left ideal $U(\g)(\u + I')$.

Let $M'$ be a finite dimensional $U(\g',e')$-module.  Then $\Ann_{U(\g',e')}(M')^\dagger$ is an
ideal of $U(\g')$ with associated variety $\overline{\cO'}$, so that $\cI_\q^\g(\Ann_{U(\g',e')}(M')^\dagger)$
is an ideal of $U(\g)$; we note that by a minor abuse of notation
we also write $\cdot^\dagger$ for the map from ideals of $U(\g',e')$ to ideals of $U(\g')$.
Also we have that $\rho_\q^\g(M') \in U(\g,e)\modfd$, so that
$\Ann_{U(\g,e)}(\rho_\q^\g(M'))^\dagger$ is a ideal of $U(\g)$ with associated variety $\overline{\cO}$.
By \cite[Corollary~6.4.2]{LoPI} there is an equality
$\cI_\q^\g(\Ann_{U(\g',e')}(M')^\dagger)=\Ann_{U(\g,e)}(\rho_\q^\g(M'))^\dagger$.
We illustrate the discussion above in the diagram below.
\begin{center}
\begin{tikzcd}
M' \arrow[r, maps to] \arrow[d, maps to]
& \rho_\q^\g(M')  \arrow[d, maps to] \\
\Ann_{U(\g',e')}(M')^\dagger \arrow[r, maps to]
 & \cI_\q^\g(\Ann_{U(\g',e')}(M')^\dagger) = \Ann_{U(\g,e)}(\rho_\q^\g(M'))^\dagger.
\end{tikzcd}
\end{center}

We make a useful observation about inducing primitive ideals.  Let $I'$ be a primitive ideal of $U(\g')$ with
associated variety $\overline {\cO'}$.
Then we can find an irreducible module $M' \in U(\g',e')\modfd$ such that $\Ann_{U(\g',e')}(M')^\dagger = I'$.
Therefore, $\cI_\q^\g(I') = \Ann_{U(\g,e)}(\rho_\q^\g(M))^\dagger$ and, in particular, it has associated variety $\overline \cO$.

We are now in a position to prove Theorem~\ref{T:noninduced}.  For the proof we no
longer consider $\g'$ to be a fixed Levi subalgebra of $\g$.

\begin{proof}[Proof of Theorem~\ref{T:noninduced}]
Let $M$ be the 1-dimensional $U(\g,e)$-module corresponding to the isolated point in $\cE^\Gamma$;
we note that this point is also an isolated point of $\cE$.  Let $I = \Ann_{U(\g,e)}(M)^\dagger$.
Then $I$ is
a multiplicity free primitive ideal of $U(\g)$ with associated variety $\overline{\cO}$

Suppose that $I$ is obtained from a primitive ideal $I'$ of $U(\g')$ by parabolic induction for some Levi
subalgebra $\g'$ of $\g$ contained in the parabolic subalgebra $\q = \g' \oplus \u$.
By the observation before this proof we see that the associated variety of $I'$
must be $\overline{\cO'}$ for some nilpotent orbit $\cO'$ in $\g'$ such that $\cO$ is obtained
from $\cO'$ by Lusztig--Spaltenstein induction.  Let $e' \in \cO'$.
Since $I'$ is primitive and has associated variety $\overline{\cO'}$, there
is a primitive ideal $J'$ of $U(\g',e')$ with finite codimension
such that $(J')^\dagger = I'$, and thus
there exists a (finite dimensional)
$U(\g',e')$-module $M'$ with $\Ann_{U(\g',e')}(M')^\dagger = I'$.
We deduce that
$$
\Ann_{U(\g,e)}(M)^\dagger = I = \cI_\q^\g(I') = \cI_\q^\g(\Ann_{U(\g',e')}(M')^\dagger) = \Ann_{U(\g,e)}(\rho_\q^\g(M'))^\dagger,
$$
where \cite[Corollary~6.4.2]{LoPI} is applied for the last equality.
In particular, this implies that $\Ann_{U(\g,e)}(M)$ and $\Ann_{U(\g,e)}(\rho_\q^\g(M'))$
are in the same $\Gamma$-orbit by \cite[Conjecture~1.2.1]{LoFD}.  Since $M$ corresponds to a point
in $\cE^\Gamma$, we deduce that $\Ann_{U(\g,e)}(M) = \Ann_{U(\g,e)}(\rho_\q^\g(M'))$, so that $M \iso \rho_\q^\g(M')$.
This implies that $M'$ is a 1-dimensional $U(\g',e')$-module, and thus we obtain a contradiction
by Proposition~\ref{P:noninduced}.  Hence, we deduce that $I$ is not induced from a primitive
ideal of $U(\g')$ for any Levi subalgebra $\g'$ of $\g$.
\end{proof}

\section{Premet's map of irreducible components}  \label{S:irrcomp}

We need to give some notation to allow us to recall \cite[Theorem~1.2]{PrCQ}.
Let $\cS_1,\dots,\cS_t$ denote the sheets of $\g$ containing $e$.
Fix an $\sl_2$-triple $(e,h,f)$ in $\g$,  write $\g^f$ for
the centralizer of $f$ in $\g$, and let $e + \g^f$ be the Slodowy slice to
the nilpotent orbit of $e$.  For $i=1, \dots, t$ we write $\cX_i := \cS_i \cap (e+\g^f)$.  For
a variety $\cX$ we write $\Comp(\cX)$ for the set of irreducible components of $\cX$.

Premet proved in \cite[Theorem~1.2]{PrCQ} that there is a surjection
$$
\tau : \Comp(\cE) \twoheadrightarrow \Comp(\cX_1) \sqcup \dots \sqcup \Comp(\cX_t)
$$
such that for any $Y \in \Comp(\cE)$ we have $\dim Y \le \dim \tau(Y)$,
and this bound on dimension is attained in each fibre of $\tau$.
We note that there is an action of $\Gamma$ on both $\Comp(\cE)$ and $\Comp(\cX_i)$ for each $i=1,\dots,t$.
One can check from the construction of $\tau$ in \cite[Section~3]{PrCQ} that it is $\Gamma$-equivariant; we note that
a slight modification is needed to the approach given in \cite{PrCQ} to work with the
definition of $U(\g,e)$ with the choice of isotropic space $\l = 0$, so that the action
of $\Gamma$ on $\cE$ can be seen.  We also recall that, by Katsylo's theorem from \cite{K},
the action of $\Gamma$ on $\Comp(\cX_i)$ is transitive.

Following the terminology of Losev in \cite[\S5.4]{LoOM} we say that $Y \in \Comp(\cE)$ is {\em large}
if $\dim Y = \dim \tau(Y)$.  It is conjectured in {\em loc.\ cit.\ }that all components
of $\cE$ are large for $\g$ of classical type, and also stated that if all components of $\cE$
are large, then $\tau$ is actually a bijection.

In the cases that we have calculated one can
verify that all irreducible components of $\cE$ are large except in the cases $(\rF_4,\rC_3(a_1))$
and $(\rE_6,\rA_3+\rA_1)$.  This is done by verifying that
\begin{itemize}
\item the number of $\Gamma$-orbits on $\Comp(\cE)$ equals the number of sheets of $\g$ containing $e$, and
\item the dimensions of components of $\cE$ in each $\Gamma$-orbit match up with the dimensions of the $\cS_i \cap (e+\g^f)$ for $i=1,\dots, t$.
\end{itemize}
Our results along with \cite[Theorems 1 and 4]{PrT}
verify the conjecture of Losev holds for classical Lie algebras with rank at most 4.  We emphasise that in
the cases $(\rF_4,\rC_3(a_1))$ and $(\rE_6,\rA_3+\rA_1)$ our calculations show that there are
non-large components of $\cE$.

\end{document}